\documentclass[journal]{IEEEtran}

\usepackage[utf8]{inputenc}
\usepackage[T1]{fontenc}
\usepackage{microtype}
\usepackage{amsmath,amssymb,amsthm,mathtools}
\usepackage{mathrsfs}
\usepackage{cite}
\usepackage{hyperref}

\hypersetup{
  hidelinks,
  pdftitle={A First-Order Entropy Law for Canonical T-Complexity of Finite-Alphabet i.i.d. Sources},
  pdfauthor={Thomas Schürmann},
  pdfkeywords={T-complexity, T-codes, canonical decomposition, i.i.d. sources, exponential integral, renewal theory}
}
\allowdisplaybreaks[1]
\theoremstyle{plain}
\newtheorem{theorem}{Theorem}[section]
\newtheorem{lemma}[theorem]{Lemma}
\newtheorem{proposition}[theorem]{Proposition}

\theoremstyle{definition}
\newtheorem{definition}[theorem]{Definition}
\newtheorem{example}[theorem]{Example}
\theoremstyle{remark}
\newtheorem{remark}[theorem]{Remark}

\newcommand{\A}{\mathcal A}
\newcommand{\F}{\mathcal F}
\newcommand{\Wcal}{\mathscr W}
\newcommand{\Psrc}{\mathbb P_{\rm src}}
\newcommand{\Esrc}{\mathbb E_{\rm src}}
\newcommand{\Pideal}{\mathbb P}
\newcommand{\Pforb}{\widetilde{\mathbb P}}
\newcommand{\Pbound}{\mathbb P_N}

\newcommand{\Eforb}{\widetilde{\mathbb E}}
\newcommand{\Paux}{\mathbb P_{\mathbf p}}
\newcommand{\Eaux}{\mathbb E_{\mathbf p}}
\newcommand{\muP}{\mu_{\mathbf p}}
\newcommand{\hp}{h(\mathbf p)}
\newcommand{\rhoP}{\rho_{\mathbf p}}
\newcommand{\alphaP}{\alpha_{\mathbf p}}
\newcommand{\lbar}{\overline\ell}
\newcommand{\one}{\mathbf 1}
\newcommand{\cT}{c_{\rm T}}

\title{A First-Order Entropy Law for Canonical T-Complexity of Finite-Alphabet i.i.d. Sources}
\author{Thomas Sch\"urmann\\\it{D\"usseldorf, Germany}%
\thanks{Thomas Sch\"urmann, Independent Researcher, 40223 D\"usseldorf, Germany}}

\begin{document}
\maketitle

\begin{abstract}
Let $W_N$ be an exact length-$N$ block from a strictly positive i.i.d.\ source $\mathbf p$ on a fixed finite alphabet.  We prove that the canonical T-complexity $\cT$ satisfies
\begin{equation*}
 \frac{\cT(W_N)}{e^{-\gamma}\hp N/\log N}\longrightarrow1
\end{equation*}
in probability and in $L^r$ for every fixed $1\le r<\infty$, where $\hp$ is the source entropy in nats and $\gamma$ is the Euler--Mascheroni constant.  The proof combines an exact length budget for canonical recovery, a critical-scale $E_1$ estimate for an ideal backward chain, and an exact finite-block boundary representation.  An exact Doob-transform identity expresses the finite-boundary law relative to the ideal law conditioned at each step to avoid the current history-dependent successor codeword.  A history-uniform renewal estimate then makes the telescoping endpoint density uniformly asymptotic to one, so no one-step approximation errors accumulate.
\end{abstract}


\section{Introduction and main result}\label{sec:intro}

Canonical T-complexity is a deterministic statistic of a finite word, obtained from the canonical generalized T-decomposition of recursively generated complete prefix codes \cite{Titchener1984,Titchener1996,Gunther1998,NicolescuTitchener1998,TitchenerEtAl2005}.  It is distinct from program-length complexity \cite{Kolmogorov1965} and from Lempel--Ziv phrase counts \cite{LempelZiv1976,ZivLempel1978}.  Exact decomposition algorithms were developed in \cite{YangSpeidel2005Jucs,YangSpeidel2005,RebenichEtAl2014}, while extremal T-complexity was studied in \cite{Speidel2008,SpeidelGulliver2012,ClarkTeutsch2015}.

Random or simple augmentation models already point to the constant $e^{-\gamma}$ \cite{HamanoYamamoto2009,Hamano2009,GulliverSpeidel2012Random}.  The present theorem concerns a different object: the pair list produced by canonical parsing of an exact fixed-length block from an arbitrary strictly positive finite-alphabet product law.  The ideal augmentation dynamics therefore cannot be used directly, because the finite left boundary and the history-dependent successor exclusion must both be retained.  We handle them by an exact boundary-chain representation and a renewal estimate uniform over all admissible histories.

Let $\A$ be a finite alphabet with $|\A|=d\ge2$, and let $\mathbf p=(p_a)_{a\in\A}$ satisfy $p_a>0$ and $\sum_a p_a=1$.  For $w=w_1\cdots w_n\in\A^+$ set
\begin{equation}\label{eq:source-weight}
 \muP(w):=\prod_{i=1}^n p_{w_i},
 \qquad
 \lambda(w):=-\log\muP(w),
\end{equation}
and write
\begin{equation}\label{eq:entropy}
 \hp:=-\sum_{a\in\A}p_a\log p_a,
 \qquad
 H_2(\mathbf p):=\frac{\hp}{\log 2}.
\end{equation}
All unsubscripted logarithms are natural, and $\gamma$ denotes the Euler--Mascheroni constant.  Let $W_N$ have the product law
\begin{equation}\label{eq:block-law}
 \Psrc\{W_N=w\}=\muP(w),\qquad w\in\A^N.
\end{equation}

\begin{theorem}[First-order law]\label{thm:main}
Let $\cT$ be the canonical T-complexity defined in Section~\ref{sec:deterministic}.  Then
\begin{equation}\label{eq:main-prob}
 \frac{\cT(W_N)}{e^{-\gamma}\hp N/\log N}
 \xrightarrow[N\to\infty]{\Psrc}1.
\end{equation}
For every fixed $1\le r<\infty$,
\begin{equation}\label{eq:main-moments}
 \Esrc\left|
 \frac{\cT(W_N)}{e^{-\gamma}\hp N/\log N}-1
 \right|^r\longrightarrow0,
\end{equation}
where $\Esrc$ denotes expectation under \eqref{eq:block-law}.  Equivalently,
\begin{equation}\label{eq:main-bits}
 \cT(W_N)=e^{-\gamma}H_2(\mathbf p)
 \frac{N}{\log_2N}\bigl(1+o_{L^r}(1)\bigr).
\end{equation}
\end{theorem}

We use $\Psrc$, $\Pideal$, $\Pforb$, and $\Pbound$ for the source, ideal, successor-conditioned, and finite-boundary laws, respectively; $\Paux$ is an auxiliary i.i.d.\ law used only in estimates.  Section~\ref{sec:deterministic} establishes canonical recovery and deterministic bounds.  Section~\ref{sec:ideal} derives the ideal length and complexity growth at the critical sequence $s_n$, Section~\ref{sec:boundary} gives the exact finite-block representation and boundary transfer, and Section~\ref{sec:main-proof} completes the proof by a hitting-index sandwich and a deterministic bound on the final boundary contribution.

\section{Canonical recovery and deterministic bounds}\label{sec:deterministic}

For a finite prefix code $C\subset\A^+$, a word $u\in C$, and $k\ge1$, define the T-augmentation
\begin{equation}\label{eq:augmentation}
 C[u,k]:=\{u^js:s\in C\setminus\{u\},\ 0\le j\le k\}
 \cup\{u^{k+1}\}.
\end{equation}
A finite list $((u_1,k_1),\ldots,(u_m,k_m))$ is a valid prescription if $C_0=\A$, $u_i\in C_{i-1}$, $k_i\ge1$, and $C_i=C_{i-1}[u_i,k_i]$.  Its \emph{arity} is the number $m$ of augmentation pairs.  Two valid prescriptions are equivalent if they generate the same final code, and a prescription is canonical if it has minimum arity in its equivalence class \cite[pp.~446--448]{TitchenerEtAl2005}.  The backward parser below also produces the local successor-exclusion property
\begin{equation}\label{eq:canonical-condition}
 u_{i+1}\ne u_i^{k_i+1},\qquad 1\le i<m.
\end{equation}
This property is derived in Proposition~\ref{prop:recovery}; it is not used as a stand-alone characterization of canonicality.

\begin{lemma}[Augmentation]\label{lem:augmentation}
If $C$ is a finite complete prefix code, then $C[u,k]$ is again a finite complete prefix code.  Its displayed representations in \eqref{eq:augmentation} are unique.  Moreover,
\begin{equation}\label{eq:augmentation-weight}
 \sum_{v\in C[u,k]}\muP(v)=1.
\end{equation}
\end{lemma}

\begin{proof}
Unique factorization over $C$ identifies the new codewords with the $C$-factor sequences $u^js$ ($s\ne u$, $0\le j\le k$) and $u^{k+1}$; none is a proper prefix of another.  If $q=\muP(u)$, then
\begin{equation*}
 \sum_{v\in C[u,k]}\muP(v)
 =(1-q)(1+q+\cdots+q^k)+q^{k+1}=1.
\end{equation*}
The same calculation with $d^{-|v|}$ gives Kraft sum one, hence completeness.
\end{proof}

\begin{definition}[Full-word recovery and T-complexity]\label{def:recovery}
Let $w\in\A^+$.  Start with $C_0=\A$ and factor $w$ over $C_0$.  Whenever the current factorization has at least two factors, let $s$ be its final factor and $u$ the factor immediately to the left of $s$.  Let $k\ge1$ be maximal such that the terminal old-factor block is $u^ks$; when $s=u$, this means a terminal run of $k+1$ copies of $u$.  Record $(u,k)$, replace $C$ by $C[u,k]$, and refactor the unchanged word.  Stop when the factorization consists of one codeword.  Set
\begin{align*}
 \pi_T(w)&=((u_1,k_1),\ldots,(u_m,k_m)),\\
 \cT(w)&:=\sum_{i=1}^m\log_2(k_i+1).
\end{align*}
\end{definition}

\begin{proposition}[Canonical decomposition interface]\label{prop:canonical-interface}
For every finite alphabet $\A$ and every $w\in\A^+$, the pair list generated by Definition~\ref{def:recovery} agrees step by step with the standard list T-decomposition parser.  It is therefore the unique canonical prescription of the unique generalized T-code for which $w$ is a maximal-length codeword.
\end{proposition}

\begin{proof}
Write $w=xa$ with $a\in\A$.  The published parser retains $a$ as a dummy tail and operates on the active token list to its left \cite[Secs.~1.1--2]{YangSpeidel2005Jucs}.  Put $U_0:=a$ and, after $i$ recorded pairs, $U_i:=u_i^{k_i}\cdots u_1^{k_1}a$.  We claim that after $i$ passes one can write $w=V_iU_i$, where the active list is exactly the ordered $C_i$-factorization of $V_i$, while the recorded pairs together with the dummy tail represent the final factor $U_i\in C_i$.  This is immediate for $i=0$.

Assume the claim after $i$ passes.  If the active list is empty, both parsers stop with $w=U_i$.  Otherwise let $u$ be its rightmost factor, and let $k\ge1$ be maximal such that the current $C_i$-factorization ends in $u^kU_i$; if $U_i=u$, this is a terminal run of $k+1$ copies.  Thus both parsers select the same pair $(u,k)$.  The list merge groups the factors strictly to the left of this terminal block into words $u^js$ with $s\ne u$ and $0\le j\le k$, or into $u^{k+1}$.  These are precisely the codewords of $C_{i+1}=C_i[u,k]$, while the terminal block becomes $U_{i+1}=u^kU_i\in C_{i+1}$.  The invariant and the step-by-step equivalence follow by induction.

This argument establishes only the parser identity.  For finite alphabets, \cite[Theorem~6.1 and Corollaries~6.3, 6.4, and~6.8]{NicolescuTitchener1998} proves that the resulting decomposition determines the unique T-code having $w$ as a maximal-length codeword and the unique minimum-arity prescription in its equivalence class; see also \cite[pp.~446--450]{TitchenerEtAl2005}.  Hence the common output is the standard canonical T-prescription.
\end{proof}

\begin{proposition}[Recovery invariants]\label{prop:recovery}
Definition~\ref{def:recovery} terminates.  If $w=xa$ with $a\in\A$, set
\begin{equation*}
 U_0:=a,
 \qquad
 U_i:=u_i^{k_i}U_{i-1}.
\end{equation*}
Then, for every $i$,
\begin{equation}\label{eq:length-budget-partial}
 U_i\in C_i,
 \qquad
 |U_i|=1+\sum_{j=1}^ik_j|u_j|
 =\max_{v\in C_i}|v|.
\end{equation}
In particular,
\begin{equation}\label{eq:exact-budget}
 |w|=1+\sum_{j=1}^m k_j|u_j|.
\end{equation}
The pair list is independent of the last source symbol:
\begin{equation}\label{eq:terminal-invariance}
 \pi_T(xa)=\pi_T(xb),\qquad a,b\in\A.
\end{equation}
More generally, every valid prescription satisfies
\begin{equation}\label{eq:valid-max-bound}
 \max_{v\in C_i}|v|\le 1+\sum_{j=1}^i k_j|u_j|,
\end{equation}
and no selected prefix can occur twice.
\end{proposition}

\begin{proof}
Suppose the current $C$-factorization has terminal form $\tau u^ks$, with $k$ maximal.  Under $C[u,k]$ the block $u^ks$ is one codeword (the distinguished word $u^{k+1}$ when $s=u$).  The preceding factors group uniquely from left to right into $u^jt$ ($0\le j\le k$, $t\ne u$) or $u^{k+1}$.  Maximality of the terminal run implies that the last new factor of the preceding part, when present, is not $u^{k+1}$.  Thus refactorization is unique, the next selected prefix satisfies \eqref{eq:canonical-condition}, and the number of factors decreases by at least one.  The algorithm therefore terminates and produces a valid prescription with the stated local reduction property.

The assertions in \eqref{eq:length-budget-partial} follow by induction.  If $U_{i-1}$ is longest in $C_{i-1}$, then $U_i=u_i^{k_i}U_{i-1}$ belongs to $C_i$, and every other new codeword has length at most $k_i|u_i|+|U_{i-1}|$.  At termination $U_m=w$, giving \eqref{eq:exact-budget}.

For terminal invariance, run the algorithm simultaneously on $xa$, $a\in\A$.  After each common number of passes, the factorization has a common ordered left-factor list followed by $U_i(a)=u_i^{k_i}\cdots u_1^{k_1}a$.  The last active factor and its terminal run are therefore independent of $a$, and the preceding list is regrouped identically.  Induction proves \eqref{eq:terminal-invariance}.

For an arbitrary valid prescription, augmentation increases the maximum codeword length by at most $k_i|u_i|$, which proves \eqref{eq:valid-max-bound} by induction.  After selecting $u$, the word $u^{k+1}$ is a current codeword, so $u$ is a proper prefix of a current codeword and cannot itself be a codeword.  Later augmentations preserve a current extension of $u$.  Hence $u$ can never be selected again.
\end{proof}

\begin{example}[A complete binary recovery]\label{ex:recovery}
For $\A=\{0,1\}$ and $w=00100$, vertical bars marking the current code factors, the recovery is
\begin{equation*}
 0|0|1|0|0
 \xrightarrow{(0,1)}00|1|00
 \xrightarrow{(1,1)}00|100
 \xrightarrow{(00,1)}00100.
\end{equation*}
The successive codes are $C_1=\{1,01,00\}$ and $C_2=\{01,101,00,100,11\}$.  The distinguished successors created at the three steps are $00$, $11$, and $0000$; the first two are the next-step exclusions.  Thus $|w|=1+1+1+2=5$ and $\cT(w)=3$.  Replacing only the last symbol gives $00101$ and the same three pairs, illustrating \eqref{eq:terminal-invariance}.
\end{example}

\begin{remark}[Implementation]\label{rem:implementation}
Definition~\ref{def:recovery} is a mathematical parser specification and need not materialize the evolving code.  The dummy-tail list representation computes the same pair sequence compactly \cite[Secs.~1.1--2]{YangSpeidel2005Jucs}; an $O(N\log N)$ implementation is analyzed in \cite[Secs.~IV--V]{YangSpeidel2005}.
\end{remark}

\begin{lemma}[Number of selected prefixes]\label{lem:prefix-count}
For a valid prescription let
\begin{equation*}
 a_n:=\#\{i:|u_i|=n\}.
\end{equation*}
Then
\begin{equation}\label{eq:prefix-count}
 a_n\le \frac{d^n}{n}+\sum_{h=1}^{\lfloor n/2\rfloor}d^h
 =\frac{d^n}{n}+O_d(d^{n/2}).
\end{equation}
\end{lemma}

\begin{proof}
Put $L_C(z)=\sum_{w\in C}z^{|w|}$.  An augmentation at length $\ell$ and factor $k$ gives
\begin{equation*}
 1-L_{C[u,k]}(z)=(1-L_C(z))(1+z^\ell+\cdots+z^{k\ell}).
\end{equation*}
Thus, for the final code $C$,
\begin{equation*}
 G(z):=\prod_i(1+z^{\ell_i}+\cdots+z^{k_i\ell_i})
 =\frac{1-L_C(z)}{1-dz},
 \qquad \ell_i=|u_i|.
\end{equation*}
All coefficient calculations are formal.  Since $\log(1-L_C(z))$ has nonpositive coefficients,
\begin{equation*}
 \beta_n:=n[z^n]\log G(z)\le d^n.
\end{equation*}
On the other hand,
\begin{equation*}
 \beta_n=\sum_{\ell_i\mid n}\ell_i
 -\sum_{(k_i+1)\ell_i\mid n}(k_i+1)\ell_i.
\end{equation*}
The first sum contains $na_n$.  Each term in the second is at most $n$ and has $\ell_i\le n/2$.  By Proposition~\ref{prop:recovery}, at most $d^h$ distinct selected words have length $h$.  Hence
\begin{equation*}
 na_n\le d^n+n\sum_{h\le n/2}d^h,
\end{equation*}
which is \eqref{eq:prefix-count}.
\end{proof}

\begin{lemma}[Uniform deterministic envelope]\label{lem:global-envelope}
There is $K_d<\infty$ such that, for every $N\ge2$ and every $w\in\A^N$,
\begin{equation}\label{eq:global-envelope}
 \cT(w)\le K_d\frac{N}{\log N}.
\end{equation}
\end{lemma}

\begin{proof}
Let $J=\lfloor\tfrac12\log_dN\rfloor$.  For selected prefixes with $|u_i|>J$,
\begin{equation*}
 \sum_{|u_i|>J}\log_2(k_i+1)
 \le\sum_{|u_i|>J}k_i
 \le\frac{N-1}{J+1}.
\end{equation*}
For $|u_i|\le J$, each term is at most $\log_2(N+1)$, while Lemma~\ref{lem:prefix-count} gives
\begin{equation*}
 \#\{i:|u_i|\le J\}=O_d(d^J/J).
\end{equation*}
The short-prefix contribution is therefore $O_d(\sqrt N)$, and the long-prefix contribution is $O_d(N/\log N)$.  Enlarging the constant handles bounded $N$.
\end{proof}

\section{The ideal backward chain}\label{sec:ideal}

For a complete prefix code $C$, an i.i.d.\ product source parsed over $C$ produces independent codewords with probabilities $q_u=\muP(u)$.  Conditional on a codeword $u$ at the penultimate position, the number of consecutive copies of $u$ ending there is geometric with parameter $1-q_u$.  Ignoring the finite left boundary and, initially, the successor exclusion therefore gives the following ideal chain.

Let
\begin{equation*}
 \Omega=(\A^+\times\mathbb N_{\ge1})^{\mathbb N},
 \qquad
 \F_n=\sigma((P_j,K_j):1\le j\le n).
\end{equation*}
Starting from $C_0=\A$, a valid coordinate prefix determines $C_n$ recursively.  Under the ideal law $\Pideal$, conditionally on $C_n$,
\begin{align}
 \Pideal(P_{n+1}=u\mid\F_n)&=q_u,\label{eq:ideal-prefix}\\
 \Pideal(K_{n+1}=k\mid\F_n,P_{n+1}=u)&=(1-q_u)q_u^{k-1},\label{eq:ideal-copy}
\end{align}
where $q_u=\muP(u)$.  Thus
\begin{equation}\label{eq:ideal-joint}
 \Pideal(P_{n+1}=u,K_{n+1}=k\mid\F_n)
 =(1-q_u)q_u^k.
\end{equation}
Completeness gives $\sum_{u\in C_n}q_u=1$.  The ideal chain generates valid prescriptions but may violate \eqref{eq:canonical-condition}; the forbidden-word conditioning in Section~\ref{sec:boundary} enforces this successor-exclusion rule.

Set $G_0=1$ and $T_0=D_0=0$.  For $m\ge1$, define
\begin{equation}\label{eq:ideal-functionals}
\begin{aligned}
 G_m&:=1+\sum_{j=1}^mK_j|P_j|,\\
 T_m&:=\sum_{j=1}^m\log_2(K_j+1),\qquad
 D_m:=T_m-m\ge0.
\end{aligned}
\end{equation}

A step is simple when $K_j=1$ and non-simple otherwise; $D_m$ is the accumulated T-complexity excess over the $m$ simple-step baseline.  Put $\rhoP=\max_a p_a<1$ and $\alphaP=\min_a(-\log p_a)>0$.

\begin{lemma}[Weight envelope]\label{lem:weight-envelope}
There is $K_{\mathbf p}<\infty$ such that, for every finite $S\subset\A^+$ with $|S|\le m$, $m\ge2$,
\begin{align}
 \sum_{w\in S}\muP(w)&\le K_{\mathbf p}\log m,\label{eq:weight-one}\\
 \sum_{w\in S}|w|\muP(w)&\le K_{\mathbf p}(\log m)^2,\label{eq:weight-two}\\
 \sum_{w\in S}\lambda(w)\muP(w)&\le K_{\mathbf p}(\log m)^2.\label{eq:weight-three}
\end{align}
\end{lemma}

\begin{lemma}[Non-simple factors]\label{lem:nonsimple}
Under $\Pideal$,
\begin{align}
 D_m&=o_{\Pideal}(m),\label{eq:complexity-excess}\\
 \sum_{j=1}^m(K_j-1)|P_j|&=o_{\Pideal}(m\log m).\label{eq:length-excess}
\end{align}
\end{lemma}

The proofs of Lemmas~\ref{lem:weight-envelope} and~\ref{lem:nonsimple} are in Appendix~\ref{app:ideal-estimates}.

For a finite complete prefix code $C$, set
\begin{equation}\label{eq:FM-def}
\begin{aligned}
 F_C(s)&:=\sum_{w\in C}\muP(w)^{1+s},\\
 M_C(z)&:=\sum_{w\in C}\muP(w)z^{|w|},\qquad
 \lbar_C:=M_C'(1).
\end{aligned}
\end{equation}
Write $F_n=F_{C_n}$, $M_n=M_{C_n}$, and $\lbar_n=\lbar_{C_n}$.  For $q_j=\muP(P_j)$ and $\ell_j=|P_j|$, put
\begin{align}
 A_j(s)&:=1+q_j^{1+s}+\cdots+q_j^{K_j(1+s)},\label{eq:A-def}\\
 Q_j(z)&:=1+q_jz^{\ell_j}+\cdots+(q_jz^{\ell_j})^{K_j}.\label{eq:Q-step-def}
\end{align}

\begin{proposition}[Product identities]\label{prop:product-identities}
For every $n$,
\begin{align}
 F_n(s)-1&=(F_0(s)-1)\prod_{j=1}^nA_j(s),\label{eq:F-product}\\
 M_n(z)-1&=(z-1)\prod_{j=1}^nQ_j(z),\label{eq:M-product}\\
 \lbar_n&=\prod_{j=1}^nA_j(0).\label{eq:lbar-product}
\end{align}
\end{proposition}

\begin{proof}
For one augmentation at $u$, with $q=\muP(u)$ and copy factor $k$, direct summation gives
\begin{equation*}
 F_{C[u,k]}(s)-1=(F_C(s)-1)(1+q^{1+s}+\cdots+q^{k(1+s)}),
\end{equation*}
and the same calculation with $q^{1+s}$ replaced by $qz^{|u|}$ gives the length-generating identity.  Iteration proves \eqref{eq:F-product} and \eqref{eq:M-product}; differentiation of \eqref{eq:M-product} at $z=1$ gives \eqref{eq:lbar-product}.
\end{proof}

For $s>0$, define the logarithmic defect
\begin{equation}\label{eq:tail-def}
 \Theta_n(s):=-\log(1-F_n(s)).
\end{equation}

For $n\ge3$, define
\begin{equation}\label{eq:critical-sequences}
 u_n:=\frac1{\log n\,\log\log n},
 \qquad
 s_n:=\frac{u_n}{\log n}
 =\frac1{(\log n)^2\log\log n}.
\end{equation}

\begin{theorem}[Critical-scale tail estimate]\label{thm:critical}
Under $\Pideal$,
\begin{equation}\label{eq:critical}
 \Theta_n(s_n)=E_1(u_n)+o_{\Pideal}(1),
 \qquad
 E_1(u):=\int_u^\infty\frac{e^{-v}}v\,dv.
\end{equation}
\end{theorem}

The proof is given in Appendix~\ref{app:critical}.  Its only special-function input is the classical small-argument relation
\begin{equation}\label{eq:E1-small}
 E_1(u)=-\gamma-\log u+o(1),\qquad u\downarrow0,
\end{equation}
see \cite[Eq.~5.1.11]{AbramowitzStegun1972}; standard concentration estimates are cited where used.

\begin{theorem}[Mean codeword length]\label{thm:mean-length}
Under $\Pideal$,
\begin{equation}\label{eq:mean-length}
 \lbar_n=\frac{e^\gamma}{\hp}\log n\bigl(1+o_{\Pideal}(1)\bigr).
\end{equation}
\end{theorem}

\begin{proof}
Theorem~\ref{thm:critical} and \eqref{eq:E1-small} give
\begin{equation*}
 1-F_n(s_n)=e^{-E_1(u_n)}\bigl(1+o_{\Pideal}(1)\bigr)
 =e^\gamma u_n\bigl(1+o_{\Pideal}(1)\bigr).
\end{equation*}
Since $1-F_0(s_n)=\hp s_n(1+o(1))$, \eqref{eq:F-product} yields
\begin{equation}\label{eq:critical-product}
 \prod_{j=1}^nA_j(s_n)
 =\frac{e^\gamma}{\hp}\log n\bigl(1+o_{\Pideal}(1)\bigr).
\end{equation}
Moreover, with $\lambda_j=-\log q_j$,
\begin{align*}
 0\le \log\frac{A_j(0)}{A_j(s)}
 &\le A_j(0)-A_j(s)\\
 &=\sum_{a=1}^{K_j}q_j^a(1-e^{-as\lambda_j})\\
 &\le s\lambda_j\sum_{a\ge1}a q_j^a
 \le \frac{s\lambda_jq_j}{(1-\rhoP)^2}.
\end{align*}
The selected prefixes are distinct, so Lemma~\ref{lem:weight-envelope} gives
\begin{equation*}
 \sum_{j=1}^n\lambda_jq_j\le K_{\mathbf p}(\log n)^2.
\end{equation*}
At $s=s_n$ the logarithm of the ratio between $\prod A_j(0)$ and $\prod A_j(s_n)$ is $O(1/\log\log n)=o(1)$.  Combine \eqref{eq:lbar-product} and \eqref{eq:critical-product}.
\end{proof}

\begin{lemma}[Summed mean length]\label{lem:summed-mean}
Under $\Pideal$,
\begin{equation}\label{eq:sum-lbar}
 \sum_{j=1}^m\lbar_{j-1}
 =\frac{e^\gamma}{\hp}m\log m\bigl(1+o_{\Pideal}(1)\bigr).
\end{equation}
\end{lemma}

\begin{proof}
By \eqref{eq:lbar-product}, $\lbar_n$ is nondecreasing pathwise.  For fixed $0<\delta<1$,
\begin{equation*}
 (1-\delta)m\,\lbar_{\lceil\delta m\rceil-1}
 \le\sum_{j=1}^m\lbar_{j-1}
 \le m\lbar_m.
\end{equation*}
Apply Theorem~\ref{thm:mean-length} and then let $\delta\downarrow0$.
\end{proof}

\begin{lemma}[Sampling fluctuation]\label{lem:sampling}
Under $\Pideal$,
\begin{equation}\label{eq:sampling}
 \sum_{j=1}^m\bigl(|P_j|-\lbar_{j-1}\bigr)
 =o_{\Pideal}(m\log m).
\end{equation}
\end{lemma}

Its proof is in Appendix~\ref{app:sampling}.

\begin{theorem}[Ideal growth law]\label{thm:ideal-growth}
Under $\Pideal$,
\begin{align}
 G_m&=\frac{e^\gamma}{\hp}m\log m\bigl(1+o_{\Pideal}(1)\bigr),\label{eq:G-growth}\\
 T_m&=m+o_{\Pideal}(m).\label{eq:T-growth}
\end{align}
\end{theorem}

\begin{proof}
Lemmas~\ref{lem:summed-mean} and~\ref{lem:sampling} give the first-order law for $\sum_{j\le m}|P_j|$; \eqref{eq:length-excess} permits replacement of $|P_j|$ by $K_j|P_j|$.  This proves \eqref{eq:G-growth}, while \eqref{eq:T-growth} is \eqref{eq:complexity-excess}.
\end{proof}

\section{Exact finite-block representation and boundary transfer}\label{sec:boundary}

For a finite complete prefix code $C$, $f\in C\cup\{\varnothing\}$, and $L\ge0$, let $\Wcal(C,f,L)$ be the set of length-$L$ words that are concatenations of $C$-codewords and, when nonempty, do not end in $f$; for $f=\varnothing$ there is no restriction.  Define the boundary mass
\begin{equation}\label{eq:B-def}
 B(C,f,L):=\sum_{y\in\Wcal(C,f,L)}\muP(y),
 \qquad B(C,f,0)=1.
\end{equation}
Write $q_f=\muP(f)$ for $f\in C$ and $q_\varnothing=0$.

\begin{lemma}[Boundary recursion]\label{lem:boundary-recursion}
For $L>0$,
\begin{equation}\label{eq:B-recursion-new}
 B(C,f,L)=\sum_{u\in C\setminus\{f\}}
 \sum_{\substack{k\ge1\\k|u|\le L}}
 q_u^k B(C[u,k],u^{k+1},L-k|u|),
\end{equation}
where the exclusion is void for $f=\varnothing$.
\end{lemma}

\begin{proof}
Every nonempty admissible context has a unique maximal terminal run $u^k$, with $u\ne f$.  Deleting that run and grouping the remaining $C$-factors according to $C[u,k]$ gives a word in $\Wcal(C[u,k],u^{k+1},L-k|u|)$.  Conversely, expanding such a word and appending $u^k$ recreates a context whose maximal terminal run is exactly $u^k$.  The bijection multiplies product weight by $q_u^k$.  Summation over the disjoint terminal-run classes gives \eqref{eq:B-recursion-new}.
\end{proof}

\begin{definition}[Exact boundary chain]\label{def:boundary-chain}
For fixed $N$, start from
\begin{equation*}
 (C_0,f_0,R_0)=(\A,\varnothing,N-1).
\end{equation*}
At a reachable state, stop if $R_j=0$.  If $R_j>0$, choose $u\in C_j\setminus\{f_j\}$ and $k\ge1$ with $k|u|\le R_j$ according to
\begin{equation}\label{eq:boundary-kernel}
 K_N((C_j,f_j,R_j);u,k)
 :=\frac{q_u^k B(C_j[u,k],u^{k+1},R_j-k|u|)}
 {B(C_j,f_j,R_j)},
\end{equation}
and set
\begin{equation}\label{eq:boundary-update}
 C_{j+1}=C_j[u,k],
 \qquad
 f_{j+1}=u^{k+1},
 \qquad
 R_{j+1}=R_j-k|u|.
\end{equation}
Lemma~\ref{lem:boundary-recursion} makes \eqref{eq:boundary-kernel} a probability kernel on reachable states.  Its law is denoted by $\Pbound$, and $\mathcal N_N:=\inf\{j\ge0:R_j=0\}$ is its number of transitions.  Under $\Pbound$, for each fixed $j$ the path functionals $G_j,T_j,D_j,R_j$ are defined on the feasible cylinder $\{j\le\mathcal N_N\}$, and an $\F_m$-event is understood as a cylinder of feasible $m$-step paths.  Accordingly, $\{R_m\ge a\}$ abbreviates $\{m\le\mathcal N_N,\ R_m\ge a\}$, and, for a bounded stopping time $\tau$, $\{R_\tau\ge a\}$ abbreviates $\{\tau\le\mathcal N_N,\ R_\tau\ge a\}$.
\end{definition}

\begin{theorem}[Exact pair law]\label{thm:exact-pair-law}
The canonical pair list of $W_N$ under $\Psrc$ has the same distribution as the pair sequence of the boundary chain under $\Pbound$, stopped when $R_j=0$.  Along either representation,
\begin{equation}\label{eq:residual-identity}
 R_j=N-1-\sum_{i=1}^jK_i|P_i|=N-G_j.
\end{equation}
\end{theorem}

\begin{proof}
Write $W_N=Ya$, where $a$ is the final source symbol.  By \eqref{eq:terminal-invariance}, the pair list depends on $Y$ but not on $a$.  Initially $Y$ has product weight $\muP(y)$ on $\A^{N-1}$, and $B(\A,\varnothing,N-1)=1$.  Inductively, after a given pair history, the unparsed left context has conditional distribution
\begin{equation}\label{eq:conditional-context}
 \frac{\muP(y)}{B(C_j,f_j,R_j)},
 \qquad y\in\Wcal(C_j,f_j,R_j).
\end{equation}
Lemma~\ref{lem:boundary-recursion} partitions these contexts by the next maximal terminal run and yields exactly the transition probability \eqref{eq:boundary-kernel}; conditioned on the transition, \eqref{eq:conditional-context} has the same form at the successor.  This proves equality of the pair laws.  Equation \eqref{eq:residual-identity} follows from \eqref{eq:exact-budget}.
\end{proof}

The boundary chain is used only as an exact probabilistic representation; recursive evaluation of $B(C,f,L)$ is not proposed as a parsing algorithm.

The comparison law $\Pforb$ uses the same code and forbidden-word updates but no finite boundary.  Its kernel is
\begin{equation}\label{eq:forbidden-kernel}
 \widetilde K((C,f);u,k)
 :=\frac{(1-q_u)q_u^k}{1-q_f},
 \qquad u\in C\setminus\{f\},\quad k\ge1.
\end{equation}
It is the ideal kernel \eqref{eq:ideal-joint} conditioned on avoiding $f$.

\begin{lemma}[Uniform ideal-to-forbidden density]\label{lem:forbidden-density}
There is $K_{\mathbf p}<\infty$ such that, for every $m$,
\begin{equation}\label{eq:forbidden-density}
 \frac{d\Pforb}{d\Pideal}\bigg|_{\F_m}
 =\prod_{j=1}^m
 \frac{\one_{\{P_j\ne f_{j-1}\}}}{1-q_{f_{j-1}}}
 \le K_{\mathbf p}
 \qquad\Pideal\text{-almost surely}.
\end{equation}
Consequently, $\Pforb(A)\le K_{\mathbf p}\Pideal(A)$ for every $A\in\F_m$.
\end{lemma}

\begin{proof}
After step $j$, $f_j=P_j^{K_j+1}$ and hence $q_{f_j}\le q_{P_j}^2$.  The selected prefixes are distinct, so
\begin{equation*}
 \sum_{j\ge1}q_{f_j}
 \le\sum_{w\in\A^+}\muP(w)^2
 =\sum_{n\ge1}\left(\sum_{a\in\A}p_a^2\right)^n<\infty.
\end{equation*}
Since $q_{f_j}\le\rhoP^2<1$, the product of $(1-q_{f_j})^{-1}$ is bounded by a deterministic constant.  The likelihood-ratio identity follows directly from conditioning the ideal kernel.
\end{proof}

Choose once and for all
\begin{equation}\label{eq:cutoff-choice}
\begin{aligned}
 1&<\mathfrak r<\rhoP^{-1},\qquad
 \eta:=\log_d(d\rhoP\mathfrak r)\in(0,1),\\
 \eta&<\vartheta<1,\qquad
 \omega_N:=\lceil N^\vartheta\rceil.
\end{aligned}
\end{equation}
Let $\mathcal C_N$ be the collection of codes obtainable from $\A$ by a valid history with length budget $\sum_i k_i|u_i|\le N-1$.

\begin{theorem}[Uniform boundary renewal]\label{thm:uniform-renewal-new}
Define
\begin{equation}\label{eq:H-def}
 \mathcal H(C,f,L):=\frac{\lbar_C B(C,f,L)}{1-q_f}.
\end{equation}
Then
\begin{equation}\label{eq:beta-def}
 \beta_N:=\sup_{\substack{C\in\mathcal C_N,\ f\in C\cup\{\varnothing\}\\
                         L\in\mathbb Z,\ \omega_N\le L\le N}}
 |\mathcal H(C,f,L)-1|\longrightarrow0.
\end{equation}
In fact, for suitable $c,K>0$ depending only on $\mathbf p,d,\mathfrak r,\vartheta$,
\begin{equation}\label{eq:beta-rate}
 \beta_N\le KN e^{-cN^\vartheta}
\end{equation}
for all sufficiently large $N$.
\end{theorem}

The proof is in Appendix~\ref{app:renewal}.

Whenever a pair path is compared with a length-$N$ block, write $R_j:=N-G_j$.

\begin{lemma}[Bulk feasibility]\label{lem:bulk-feasibility}
For all sufficiently large $N$, every $\widetilde K$-admissible $m$-step path satisfying $R_m\ge\omega_N$ is feasible for the boundary chain, and $B(C_j,f_j,R_j)>0$ for $0\le j\le m$.  The same holds up to a bounded stopping time $\tau$ on $\{R_\tau\ge\omega_N\}$.
\end{lemma}

\begin{proof}
Along such a path, $R_j\ge R_m\ge\omega_N$ and $C_j\in\mathcal C_N$.  Since $\beta_N<1$ for large $N$, Theorem~\ref{thm:uniform-renewal-new} gives $\mathcal H(C_j,f_j,R_j)>0$, hence $B(C_j,f_j,R_j)>0$.  For $j<m$, monotonicity gives $R_{j+1}\ge R_m>0$, hence $K_{j+1}|P_{j+1}|\le R_j$.  The stopping-time statement follows pathwise.
\end{proof}

\begin{proposition}[Exact Doob transform]\label{prop:doob}
Write $\mathcal H(s)=\mathcal H(C,f,L)$ for a state $s=(C,f,L)$, and let
\begin{equation*}
 s'=(C[u,k],u^{k+1},L-k|u|)
\end{equation*}
be a feasible boundary transition.  Then
\begin{equation}\label{eq:doob-kernel}
 K_N(s;u,k)=\widetilde K((C,f);u,k)
 \frac{\mathcal H(s')}{\mathcal H(s)}.
\end{equation}
Consequently, for all sufficiently large $N$, if $A\in\F_m$ and $A\subset\{R_m\ge\omega_N\}$, then
\begin{equation}\label{eq:doob-path}
 \Pbound(A)=\Eforb\bigl[\one_A\mathcal H(C_m,f_m,R_m)\bigr]
\end{equation}
and
\begin{equation}\label{eq:bulk-event-transfer}
 |\Pbound(A)-\Pforb(A)|\le\beta_N\Pforb(A).
\end{equation}
If $\tau$ is bounded, the same statements hold for $A\in\F_\tau$ with $A\subset\{R_\tau\ge\omega_N\}$, with $m$ replaced by $\tau$.
\end{proposition}

\begin{proof}
The mean-length update obtained from \eqref{eq:M-product} is
\begin{equation}\label{eq:lbar-update}
 \lbar_{C[u,k]}=\lbar_C(1+q_u+\cdots+q_u^k)
 =\lbar_C\frac{1-q_u^{k+1}}{1-q_u}.
\end{equation}
Substituting \eqref{eq:H-def}, \eqref{eq:boundary-kernel}, and \eqref{eq:forbidden-kernel} gives \eqref{eq:doob-kernel} exactly.  By Lemma~\ref{lem:bulk-feasibility}, every path in the stated events is feasible.  Along a length-$m$ path the product of the ratios in \eqref{eq:doob-kernel} telescopes.  Since $\lbar_{\A}=1$, $B(\A,\varnothing,N-1)=1$, and $q_{\varnothing}=0$,
\begin{equation*}
 \mathcal H(\A,\varnothing,N-1)=1.
\end{equation*}
Thus the path density is $\mathcal H(C_m,f_m,R_m)$, proving \eqref{eq:doob-path}.  On $R_m\ge\omega_N$, Theorem~\ref{thm:uniform-renewal-new} gives \eqref{eq:bulk-event-transfer}.  For bounded $\tau$ and $A\in\F_\tau$, apply the deterministic-time identity to the disjoint events $A\cap\{\tau=j\}$.
\end{proof}

\section{Proof of Theorem~\ref{thm:main}}\label{sec:main-proof}

For all sufficiently large $N$, work under the exact boundary law $\Pbound$ and use $\mathcal N_N$ from Definition~\ref{def:boundary-chain}.  Put
\begin{equation}\label{eq:kappa}
 \kappa_N:=e^{-\gamma}\hp\frac{N}{\log N},
 \qquad
 x_N:=N-\omega_N,
\end{equation}
and define the last bulk index
\begin{equation}\label{eq:sigma}
 \sigma_N:=\max\{0\le j\le\mathcal N_N:R_j\ge\omega_N\}.
\end{equation}
For fixed $0<\delta<1$, set
\begin{equation}\label{eq:m-pm}
 m_-:=\lfloor(1-\delta)\kappa_N\rfloor,
 \qquad
 m_+:=\lceil(1+\delta)\kappa_N\rceil.
\end{equation}

By Theorem~\ref{thm:ideal-growth}, $\log m_\pm/\log N\to1$ and
\begin{equation*}
 \frac{e^\gamma}{\hp}\frac{m_\pm\log m_\pm}{N}
 \longrightarrow1\pm\delta,
 \qquad
 \frac{x_N}{N}\longrightarrow1.
\end{equation*}
Consequently,
\begin{equation}\label{eq:ideal-sandwich}
 \Pideal\{G_{m_-}>x_N\}\longrightarrow0,
 \qquad
 \Pideal\{G_{m_+}\le x_N\}\longrightarrow0.
\end{equation}
Lemma~\ref{lem:forbidden-density} transfers these limits to $\Pforb$.  Let $A_m=\{G_m\le x_N\}=\{R_m\ge\omega_N\}$.  Proposition~\ref{prop:doob} gives $\Pbound(A_{m_-})\to1$ and $\Pbound(A_{m_+})\to0$: for the first limit apply \eqref{eq:bulk-event-transfer} to $A_{m_-}$ and take complements only afterward, while the second limit follows directly.  Since $A_m=\{\sigma_N\ge m\}$ under the convention of Definition~\ref{def:boundary-chain},
\begin{equation}\label{eq:sigma-sandwich}
 \Pbound\{m_-\le\sigma_N<m_+\}\longrightarrow1.
\end{equation}

It remains to control the non-simple complexity excess at the random bulk index.  For $\varepsilon>0$ and $1\le j\le m_+$, let $E_{j,N}$ be the feasible $j$-step cylinder
\begin{equation*}
 E_{j,N}:=\{D_{j-1}\le\varepsilon\kappa_N<D_j,\ R_j\ge\omega_N\}.
\end{equation*}
Since $D_j$ is nondecreasing, the events $E_{j,N}$ are disjoint and
\begin{align}
 \Pbound\{D_{\sigma_N}>\varepsilon\kappa_N\}
 &\le \Pbound\{\sigma_N\ge m_+\}
      +\sum_{j=1}^{m_+}\Pbound(E_{j,N}).
 \label{eq:D-sigma-bound}
\end{align}
Each $E_{j,N}$ belongs to $\F_j$ and is contained in $\{R_j\ge\omega_N\}$.  Proposition~\ref{prop:doob} and Lemma~\ref{lem:forbidden-density} therefore give
\begin{align*}
 \sum_{j=1}^{m_+}\Pbound(E_{j,N})
 &\le(1+\beta_N)\sum_{j=1}^{m_+}\Pforb(E_{j,N})\\
 &\le(1+\beta_N)\Pforb\{D_{m_+}>\varepsilon\kappa_N\}\\
 &\le K_{\mathbf p}(1+\beta_N)
      \Pideal\{D_{m_+}>\varepsilon\kappa_N\}.
\end{align*}
The last probability tends to zero by \eqref{eq:T-growth}, since $m_+/\kappa_N\to1+\delta$.  Together with \eqref{eq:sigma-sandwich}, this proves
\begin{equation}\label{eq:D-sigma}
 D_{\sigma_N}=o_{\Pbound}(\kappa_N).
\end{equation}
Because $T_{\sigma_N}=\sigma_N+D_{\sigma_N}$, equations \eqref{eq:sigma-sandwich} and \eqref{eq:D-sigma}, followed by $\delta\downarrow0$, give
\begin{equation}\label{eq:bulk-complexity}
 T_{\sigma_N}=\kappa_N\bigl(1+o_{\Pbound}(1)\bigr).
\end{equation}

The contribution of the remaining boundary segment admits a deterministic bound.  If $\sigma_N<\mathcal N_N$, the next pair crosses below the cutoff.  Its copy factor is at most $N$, hence its complexity is at most $\log_2(N+1)$.  After that crossing the remaining length budget is below $\omega_N$, and
\begin{equation*}
 \sum\log_2(k_i+1)\le\sum k_i
 \le\sum k_i|u_i|<\omega_N.
\end{equation*}
Therefore
\begin{equation}\label{eq:boundary-remainder}
 0\le T_{\mathcal N_N}-T_{\sigma_N}
 \le\log_2(N+1)+\omega_N
 =o(\kappa_N).
\end{equation}
By Theorem~\ref{thm:exact-pair-law}, $T_{\mathcal N_N}$ has the same distribution as $\cT(W_N)$.  Thus \eqref{eq:bulk-complexity}--\eqref{eq:boundary-remainder} prove \eqref{eq:main-prob}.

Finally, Lemma~\ref{lem:global-envelope} implies that
\begin{equation*}
 0\le \frac{\cT(w)}{e^{-\gamma}\hp N/\log N}\le K_{\mathbf p,d}
\end{equation*}
uniformly over $w\in\A^N$ and all sufficiently large $N$.  Uniform boundedness together with convergence in probability gives \eqref{eq:main-moments}; \eqref{eq:main-bits} is the identity $\hp/\log N=H_2(\mathbf p)/\log_2N$.

\section{Discussion}\label{sec:discussion}

The first-order source dependence is entirely through the entropy $\hp$, whereas $e^{-\gamma}$ is universal for the ideal augmentation mechanism.  The exact boundary representation explains why that ideal constant survives canonical parsing of a fixed-length block: the finite boundary changes the bulk law by one telescoping density rather than by a product of uncontrolled local errors.

The result does not provide a rate, an almost-sure limit, concentration, or a second-order term.  These would require quantitative control of the critical-scale remainder and sharper fluctuations of sampled codeword lengths.  Extending the argument to finite-state sources appears to require a boundary renewal theorem uniform simultaneously in the evolving T-code and the source state; zero probabilities or a growing alphabet would introduce additional nonuniformities.

\appendices

\section{Ideal-chain estimates}\label{app:ideal-estimates}

Throughout the appendices, constants denoted by $K$ may change from line to line and depend only on the fixed source unless further parameters are displayed.

\begin{proof}[Proof of Lemma~\ref{lem:weight-envelope}]
Let $X_1,X_2,\ldots$ be i.i.d.\ with law $\mathbf p$ on an auxiliary space with probability $\Paux$ and expectation $\Eaux$.  Put $Y_i=-\log p_{X_i}$ and $\Sigma_n=Y_1+\cdots+Y_n$.  Since $Y_i\ge\alphaP$, for $x\ge1$,
\begin{align*}
 \sum_{\lambda(w)\le x}\muP(w)
 &=\sum_{n\ge1}\Paux\{\Sigma_n\le x\}
 \le x/\alphaP+1,\\
 \sum_{\lambda(w)\le x}|w|\muP(w)
 &\le (x/\alphaP+1)^2,\\
 \sum_{\lambda(w)\le x}\lambda(w)\muP(w)
 &\le x(x/\alphaP+1).
\end{align*}
For a set $S$ with $|S|\le m$, split at $x=2\log m$.  The displayed estimates control the low-information part.  On $\lambda(w)>x$, each word satisfies
\begin{equation*}
 \muP(w)\le e^{-x}=m^{-2},
 \qquad
 |w|\muP(w)\le\alphaP^{-1}\lambda(w)e^{-\lambda(w)},
\end{equation*}
and the function $t e^{-t}$ is decreasing for $t\ge1$.  Thus the three tail contributions are bounded respectively by
\begin{equation*}
 m e^{-x},
 \qquad
 \alphaP^{-1}mxe^{-x},
 \qquad
 mxe^{-x}.
\end{equation*}
Combining the two ranges proves \eqref{eq:weight-one}--\eqref{eq:weight-three}.
\end{proof}

\begin{proof}[Proof of Lemma~\ref{lem:nonsimple}]
Conditional on $P_j=u$, with $q=\muP(u)$,
\begin{align*}
 &\mathbb E\bigl[\log_2(K_j+1)-1\mid\F_{j-1},P_j=u\bigr]\\
 &\qquad=(1-q)\sum_{k\ge2}q^{k-1}
       \bigl(\log_2(k+1)-1\bigr)
 \le Kq,
\end{align*}
where the last bound is uniform for $q\le\rhoP<1$, and
\begin{equation*}
 \mathbb E(K_j-1\mid\F_{j-1},P_j=u)=\frac q{1-q}\le Kq.
\end{equation*}
The selected prefixes are distinct.  Lemma~\ref{lem:weight-envelope} therefore gives the pathwise bounds
\begin{equation*}
 \sum_{j\le m}\muP(P_j)\le K\log m,
 \qquad
 \sum_{j\le m}|P_j|\muP(P_j)\le K(\log m)^2.
\end{equation*}
Taking expectations and applying Markov's inequality proves \eqref{eq:complexity-excess} and \eqref{eq:length-excess}.
\end{proof}

\section{The critical one-scale estimate}\label{app:critical}

Let $\mathcal T$ be the shift $(\mathcal Tf)(s)=f(s+1)$ and define the deterministic linear recursion
\begin{equation}\label{eq:linear-recursion}
 \Lambda_0(s):=-\log(1-F_0(s)),
 \qquad
 \Lambda_n:=(I-\mathcal T)^n\Lambda_0.
\end{equation}
On the auxiliary i.i.d.\ source space used in the proof of Lemma~\ref{lem:weight-envelope},
\begin{equation}\label{eq:linear-renewal-representation}
 \Lambda_n(s)=\sum_{r\ge1}\frac1r
 \Eaux\left[e^{-s\Sigma_r}(1-e^{-\Sigma_r})^n\right].
\end{equation}
Indeed, expand $-\log(1-F_0(s))=\sum_{r\ge1}F_0(s)^r/r$ and apply $(I-\mathcal T)^n$ termwise.

\begin{lemma}[Linear critical scale]\label{lem:linear-critical}
For $u_n,s_n$ from \eqref{eq:critical-sequences},
\begin{equation}\label{eq:linear-critical}
 \Lambda_n(s_n)=E_1(u_n)+o(1).
\end{equation}
\end{lemma}

\begin{proof}
Put $x=\log n$, $\delta=x^{-1/4}$, and
\begin{equation*}
 r^-:=\left\lfloor\frac{(1-\delta)x}{\hp}\right\rfloor,
 \qquad
 r^+:=\left\lceil\frac{(1+\delta)x}{\hp}\right\rceil.
\end{equation*}
The bounded positive variables $Y_i$ satisfy the standard Bernstein estimate \cite[Chap.~2]{BoucheronLugosiMassart2013}
\begin{equation}\label{eq:bernstein}
 \Paux\{|\Sigma_r-r\hp|>t\}
 \le K\exp\{-c\min(t^2/r,t)\}.
\end{equation}
For $r\le r^-$, split according to
$\Sigma_r\le(1-\delta/2)x$.  On this event,
\begin{equation*}
 (1-e^{-\Sigma_r})^n\le\exp\{-n^{\delta/2}\},
\end{equation*}
and the complementary probability is at most $Ke^{-c\delta^2x}$.  Since $\sum_{r\le r^-}r^{-1}=O(\log x)$, the contribution of $r\le r^-$ is $o(1)$.  The transition window $r^-<r<r^+$ contributes $O(\delta)=o(1)$.

For $r\ge r^+$, let $G_r=\{\Sigma_r\ge(1+\delta/2)x\}$.  On $G_r$,
\begin{equation*}
 0\le1-(1-e^{-\Sigma_r})^n\le n^{-\delta/2}.
\end{equation*}
Moreover, for all sufficiently large $n$ and every $r\ge r^+$,
\begin{equation*}
 r\hp-(1+\delta/2)x\ge c\delta r.
\end{equation*}
Thus \eqref{eq:bernstein} gives
\begin{equation*}
 \Paux(G_r^c)\le Ke^{-c\delta^2r}.
\end{equation*}
Let
\begin{equation*}
 \beta(s):=-\log\Eaux e^{-sY_1}=\hp s+O(s^2).
\end{equation*}
Then the replacement of the smooth cutoff by one on $r\ge r^+$ has error bounded by
\begin{equation*}
 n^{-\delta/2}\sum_{r\ge r^+}\frac{e^{-\beta(s_n)r}}r
 +K\sum_{r\ge r^+}\frac{e^{-c\delta^2r}}r=o(1).
\end{equation*}
The first sum is $O(\log(1/u_n))=O(\log x)$, while the second is
$O(E_1(c\delta^2r^+))=o(1)$.

Finally, monotonicity of $t\mapsto e^{-\beta t}/t$ gives
\begin{equation*}
 \sum_{r\ge r^+}\frac{e^{-\beta(s_n)r}}r
 =E_1(\beta(s_n)r^+)+O((r^+)^{-1}).
\end{equation*}
Since $s_nx=u_n$,
\begin{equation*}
 \beta(s_n)r^+=u_n(1+o(1)).
\end{equation*}
For $|\varepsilon|\le1/2$,
\begin{equation*}
 |E_1(u(1+\varepsilon))-E_1(u)|
 \le\int_{u(1-|\varepsilon|)}^{u(1+|\varepsilon|)}\frac{dt}{t}
 \le K|\varepsilon|,
\end{equation*}
uniformly for $u>0$.  Thus the last display equals $E_1(u_n)+o(1)$.  Together with the lower and transition ranges, this proves \eqref{eq:linear-critical}.
\end{proof}

The nonlinear comparison is conveniently expressed with absolutely summable exponential series.  Let $\mathcal D$ be the additive semigroup generated by $\{-\log p_a:a\in\A\}$, including zero.  For
\begin{equation*}
 f(s)=\sum_{\nu\in\mathcal D}c_\nu e^{-\nu s},
 \qquad
 \|f\|_s:=\sum_{\nu\in\mathcal D}|c_\nu|e^{-\nu s},
\end{equation*}
all equal exponents are grouped.

\begin{lemma}[Exponential-series damping]\label{lem:damping}
For fixed $s\ge0$, the series with finite $\|\cdot\|_s$-norm form a Banach algebra.  In particular, coefficient convolution gives $\|fg\|_s\le\|f\|_s\|g\|_s$.  If $\|g\|_s<1$, then $-\log(1-g)=\sum_{r\ge1}g^r/r$ converges in this norm and
\begin{equation*}
 \|{-\log(1-g)}\|_s\le-\log(1-\|g\|_s).
\end{equation*}
For $k=0$, $|[(I-\mathcal T)^0f](s)|=|f(s)|\le\|f\|_s$.  For every $k\ge1$,
\begin{equation}\label{eq:damping-bound}
 \bigl|[(I-\mathcal T)^kf](s)\bigr|
 \le\sum_\nu|c_\nu|(1-e^{-\nu})^ke^{-\nu s}
 \le\|f\|_s.
\end{equation}
If $s_n\ge0$, $s_n\to0$, and $\|f\|_0<\infty$, then
\begin{equation}\label{eq:damping-limit}
 [(I-\mathcal T)^nf](s_n)\longrightarrow0.
\end{equation}
For a random such series with measurable coefficients, $\mathbb E\|f\|_0^r<\infty$, $r\in\{1,2\}$, implies that the convergence in \eqref{eq:damping-limit} also holds in $L^r$.  If $\mathbb E\|f\|_s<\infty$, coefficientwise conditional expectation is well defined and satisfies
\begin{equation}\label{eq:conditional-norm}
 \bigl\|\mathbb E(f\mid\mathcal G)\bigr\|_s
 \le\mathbb E(\|f\|_s\mid\mathcal G),
\end{equation}
and it commutes with evaluation and with $(I-\mathcal T)^k$.
\end{lemma}

\begin{proof}
The coefficient space is a weighted $\ell^1(\mathcal D)$ space and hence complete.  The semigroup $\mathcal D$ is countable and locally finite, so convolution and Tonelli give the product bound and the logarithmic-series assertion.  For $k\ge1$,
\begin{equation*}
 (I-\mathcal T)^ke^{-\nu t}=(1-e^{-\nu})^ke^{-\nu t}.
\end{equation*}
Thus \eqref{eq:damping-bound} follows, and dominated convergence proves \eqref{eq:damping-limit} and its $L^r$ versions.  Scalar conditional expectations of the countably many coefficients, followed by conditional Tonelli and finite truncation, give \eqref{eq:conditional-norm} and the commutation properties.
\end{proof}

Put
\begin{equation}\label{eq:S-r}
 S_{r,j}:=\sum_{u\in C_j}\muP(u)^r.
\end{equation}

\begin{lemma}[Summable nonlinear scale]\label{lem:S3}
\begin{equation}\label{eq:S3-tail}
 \mathbb E\sum_{j\ge J}S_{3,j}\longrightarrow0
 \qquad(J\to\infty).
\end{equation}
\end{lemma}

\begin{proof}
At step $j+1$, after selecting $u$ with weight $q$, write $k:=K_{j+1}$ and let
\begin{equation*}
 \Delta_{j+1}:=\sum_{a=1}^{k}q^{2a}.
\end{equation*}
The product identity \eqref{eq:F-product} at $s=1$ gives
\begin{equation*}
 S_{2,j+1}=S_{2,j}-(1-S_{2,j})\Delta_{j+1}.
\end{equation*}
Since $S_{2,j}\le S_{2,0}<1$ and
\begin{equation*}
 \mathbb E(\Delta_{j+1}\mid C_j)
 =\sum_{u\in C_j}\frac{q_u^3}{1-q_u^3}\ge S_{3,j},
\end{equation*}
one obtains
\begin{equation}\label{eq:S2-drift}
 \mathbb E(S_{2,j+1}\mid C_j)
 \le S_{2,j}-c_{\mathbf p}S_{3,j}.
\end{equation}
Also $S_{2,j}^2\le S_{3,j}$ by Cauchy--Schwarz and $\sum q_u=1$.  Hence
$\mathbb E S_{3,j}\ge\mathbb E(S_{2,j}^2)\ge(\mathbb E S_{2,j})^2$.  With $a_j:=\mathbb E S_{2,j}$, \eqref{eq:S2-drift} therefore gives
\begin{equation*}
 a_{j+1}\le a_j-c_{\mathbf p}a_j^2,
\end{equation*}
so $a_j\to0$.  Summing \eqref{eq:S2-drift} from $J$ to $n$ and then letting $n\to\infty$ gives
\begin{equation*}
 c_{\mathbf p}\mathbb E\sum_{j\ge J}S_{3,j}
 \le a_J,
\end{equation*}
which proves \eqref{eq:S3-tail}.
\end{proof}

Let
\begin{equation*}
 \xi_{j+1}(s):=\log A_{j+1}(s),
 \qquad
 \varepsilon_{j+1}(s):=\xi_{j+1}(s)-\Theta_j(s+1).
\end{equation*}
Since $\Theta_{j+1}(s)=\Theta_j(s)-\xi_{j+1}(s)$, iteration gives
\begin{equation}\label{eq:rest-decomposition-new}
 \Theta_n(s)-\Lambda_n(s)
 =-\sum_{j=0}^{n-1}[(I-\mathcal T)^{n-1-j}\varepsilon_{j+1}](s).
\end{equation}

\begin{lemma}[Local nonlinear error]\label{lem:local-error-new}
With coefficientwise conditional expectation, write
\begin{equation*}
 b_j:=\mathbb E(\xi_{j+1}\mid\F_j)-\mathcal T\Theta_j,
 \qquad
 \zeta_{j+1}:=\xi_{j+1}-\mathbb E(\xi_{j+1}\mid\F_j).
\end{equation*}
Then, for every $s\ge0$,
\begin{equation}\label{eq:local-bounds-new}
 \|b_j\|_s\le KS_{3,j},
 \qquad
 \mathbb E(\|\zeta_{j+1}\|_s^2\mid\F_j)\le KS_{3,j}.
\end{equation}
\end{lemma}

\begin{proof}
Condition on $C_j$ and on selecting $u$, and write $q=\muP(u)$, $k:=K_{j+1}$, and $x(t)=q^{1+t}$.  Since $q\le\rhoP<1$,
\begin{equation}\label{eq:xi-expansion-new}
 \xi_{j+1}(t)
 =\sum_{r\ge1}\frac{x(t)^r}{r}
 -\sum_{r\ge1}\frac{x(t)^{r(k+1)}}r
\end{equation}
converges absolutely in the series norm.  Conditional on $u$,
\begin{equation*}
 \mathbb E[x(t)^{r(k+1)}\mid u]
 =\frac{(1-q)x(t)^{2r}}{1-qx(t)^r}.
\end{equation*}
After averaging over the selected prefix, the $r=1$ term of the first sum is $F_j(t+1)$.  Its remaining terms satisfy
\begin{equation*}
 \sum_{u\in C_j}q_u\sum_{r\ge2}\frac{q_u^{r(1+s)}}r
 \le K\sum_{u\in C_j}q_u^3=KS_{3,j}.
\end{equation*}
For the second sum, expand the denominator geometrically.  Since
$q_u^{1+r(1+s)}\le\rhoP^2$, uniformly in $r\ge1$ and $s\ge0$,
\begin{align*}
 &\sum_{u\in C_j}q_u(1-q_u)\sum_{r\ge1}
 \frac{q_u^{2r(1+s)}}{r(1-q_u^{1+r(1+s)})}\\
 &\hspace{35mm}\le K\sum_{u\in C_j}q_u^3=KS_{3,j}.
\end{align*}
Furthermore $\|F_j(\cdot+1)\|_s=\sum_uq_u^{2+s}\le S_{2,j}\le\rhoP<1$.  The logarithmic series and Cauchy--Schwarz therefore give
\begin{align*}
 \|\Theta_j(\cdot+1)-F_j(\cdot+1)\|_s
 &\le K\left(\sum_uq_u^{2+s}\right)^2\\
 &\le K\sum_uq_u^{3+2s}\le KS_{3,j}.
\end{align*}
These estimates prove the bias bound.  Finally, \eqref{eq:xi-expansion-new} gives $\|\xi_{j+1}\|_s\le Kq^{1+s}$ uniformly in the copy factor.  Averaging its square with selection probability $q$ yields $K\sum_uq_u^{3+2s}\le KS_{3,j}$; conditional Jensen gives the same bound for the centered variable $\zeta_{j+1}$.
\end{proof}

\begin{lemma}[Damped nonlinear remainder]\label{lem:nonlinear-remainder}
For every deterministic $t_n>0$ with $t_n\to0$,
\begin{equation}\label{eq:nonlinear-remainder}
 \Theta_n(t_n)-\Lambda_n(t_n)
 \xrightarrow[n\to\infty]{\Pideal}0.
\end{equation}
\end{lemma}

\begin{proof}
Insert $\varepsilon_{j+1}=b_j+\zeta_{j+1}$ in \eqref{eq:rest-decomposition-new}.  Fix $J$.  For each $j<J$, Lemma~\ref{lem:damping} makes the damped term tend to zero in $L^1$ for $b_j$ and in $L^2$ for $\zeta_{j+1}$.  For the predictable tail, Lemmas~\ref{lem:damping} and~\ref{lem:local-error-new} give
\begin{equation*}
 \mathbb E\left|
 \sum_{j=J}^{n-1}[(I-\mathcal T)^{n-1-j}b_j](t_n)
 \right|
 \le K\mathbb E\sum_{j\ge J}S_{3,j}.
\end{equation*}
Set $Z_{j,n}=[(I-\mathcal T)^{n-1-j}\zeta_{j+1}](t_n)$.  Then $Z_{j,n}$ is $\F_{j+1}$-measurable and, because coefficientwise conditional expectation commutes with the damping operator,
$\mathbb E(Z_{j,n}\mid\F_j)=0$.  If $i<j$, then $Z_{i,n}$ is $\F_j$-measurable, so $\mathbb E(Z_{i,n}Z_{j,n})=0$.  Thus the terms are pairwise orthogonal and
\begin{equation*}
 \mathbb E\left|
 \sum_{j=J}^{n-1}Z_{j,n}
 \right|^2
 \le K\mathbb E\sum_{j\ge J}S_{3,j}.
\end{equation*}
First let $n\to\infty$ for the finite head and then $J\to\infty$, using Lemma~\ref{lem:S3}.
\end{proof}

\begin{proof}[Proof of Theorem~\ref{thm:critical}]
Apply Lemma~\ref{lem:nonlinear-remainder} with $t_n=s_n$ and combine it with Lemma~\ref{lem:linear-critical}.
\end{proof}

\section{Sampling fluctuation}\label{app:sampling}

\begin{lemma}[Discrete Lenglart inequality]\label{lem:lenglart}
Let $(M_k)_{k=0}^m$ be a square-integrable martingale adapted to $(\mathcal G_k)_{k=0}^m$, with $M_0=0$, and set
\begin{equation*}
 V_k:=\sum_{j=1}^k
 \mathbb E[(M_j-M_{j-1})^2\mid\mathcal G_{j-1}].
\end{equation*}
Then, for $a,b>0$,
\begin{equation}\label{eq:lenglart}
 \mathbb P\left\{\max_{k\le m}|M_k|>a\right\}
 \le\frac b{a^2}+\mathbb P\{V_m>b\}.
\end{equation}
\end{lemma}

\begin{proof}
Let
\begin{align*}
 \tau&:=\inf\{k\le m:|M_k|>a\},\\
 \sigma&:=\inf\{0\le k<m:V_{k+1}>b\},
\end{align*}
with the infimum of an empty set equal to $m$, and put $\theta=\tau\wedge\sigma\wedge m$.  Predictability of $V$ makes $\sigma$ a stopping time, and $V_\theta\le b$.  Since $M_k^2-V_k$ is a martingale, optional stopping gives $\mathbb E M_\theta^2=\mathbb E V_\theta\le b$.  On $\{\max_{k\le m}|M_k|>a,\ V_m\le b\}$ one has $|M_\theta|>a$.  Markov's inequality therefore gives probability at most $b/a^2$.  Adding $\mathbb P\{V_m>b\}$ proves \eqref{eq:lenglart}.
\end{proof}

\begin{proof}[Proof of Lemma~\ref{lem:sampling}]
Put
\begin{equation*}
 \mathcal M_m:=\sum_{j=1}^m(|P_j|-\lbar_{j-1}),
 \qquad
 J_n:=\sum_{u\in C_n}\muP(u)|u|^2.
\end{equation*}
The process $(\mathcal M_m)_{m\ge0}$ is a square-integrable martingale.  From \eqref{eq:M-product}, writing
$M_n(z)-1=(z-1)\mathcal Q_n(z)$ with $\mathcal Q_n(z):=\prod_{i\le n}Q_i(z)$, one obtains
\begin{equation*}
 J_n=\lbar_n\left(1+2\sum_{i=1}^n\frac{Q_i'(1)}{Q_i(1)}\right).
\end{equation*}
Furthermore,
\begin{equation*}
 \frac{Q_i'(1)}{Q_i(1)}
 \le |P_i|\sum_{a\ge1}a q_i^a
 \le K|P_i|q_i.
\end{equation*}
For $n\ge2$, the selected prefixes are distinct, so Lemma~\ref{lem:weight-envelope} gives
\begin{equation*}
 \sum_{i=1}^n|P_i|q_i\le K(\log(n+2))^2.
\end{equation*}
Together with the preceding display, and after enlarging $K$ for $n=0,1$, this yields
\begin{equation}\label{eq:J-bound-new}
 J_n\le\lbar_n\bigl(1+K(\log(n+2))^2\bigr),
 \qquad n\ge0.
\end{equation}
Let $V_m$ be the predictable quadratic variation of $\mathcal M_m$.  Since each conditional variance is at most $J_{j-1}$, \eqref{eq:J-bound-new} and \eqref{eq:sum-lbar} give
\begin{align*}
 V_m
 &\le K(\log(m+1))^2\sum_{j=1}^m\lbar_{j-1}\\
 &=O_{\Pideal}(m(\log m)^3)
 =o_{\Pideal}(m^2(\log m)^2).
\end{align*}
Apply Lemma~\ref{lem:lenglart} with $a=\varepsilon m\log m$ and
$b=A m(\log m)^3$, then let first $m\to\infty$ and subsequently $A\to\infty$.
\end{proof}

\section{Uniform finite-block renewal}\label{app:renewal}

Recall $M_C(z)=\sum_{w\in C}\muP(w)z^{|w|}$.  Choose any valid history generating $C\in\mathcal C_N$, with selected lengths $\ell_i$, weights $q_i$, and copy factors $k_i$.  Then \eqref{eq:M-product} gives
\begin{equation}\label{eq:code-factor}
\begin{aligned}
 M_C(z)-1&=(z-1)\mathcal Q_C(z),\\
 \mathcal Q_C(z)&:=\prod_i
 \bigl(1+q_iz^{\ell_i}+\cdots+(q_iz^{\ell_i})^{k_i}\bigr).
\end{aligned}
\end{equation}
Moreover, $\mathcal Q_C(1)=\lbar_C$.  Although the product is written through a history, $\mathcal Q_C=(M_C-1)/(z-1)$ depends only on the final code; any generating history may therefore be used for uniform bounds.

\begin{lemma}[Context generating function]\label{lem:context-gf-new}
Let
\begin{equation*}
 \phi_f(z):=
 \begin{cases}
  \muP(f)z^{|f|},&f\in C,\\
  0,&f=\varnothing.
 \end{cases}
\end{equation*}
Then
\begin{equation}\label{eq:context-gf-new}
 \sum_{L\ge0}B(C,f,L)z^L
 =\frac{1-\phi_f(z)}{1-M_C(z)}.
\end{equation}
\end{lemma}

\begin{proof}
Unique factorization over the prefix code $C$ makes the generating function of all finite $C$-sequences equal to $(1-M_C(z))^{-1}$.  Sequences ending in $f$ have generating function $\phi_f(z)/(1-M_C(z))$.  Subtraction proves the claim; for $f=\varnothing$ there is no subtraction.
\end{proof}

\begin{lemma}[Uniform zero-free control]\label{lem:zero-free-new}
Put $a_0:=\rhoP\mathfrak r<1$.  For $|z|=\mathfrak r$,
\begin{equation}\label{eq:Q-inverse-new}
 |\mathcal Q_C(z)|^{-1}\le
 \exp\left(K\sum_i a_0^{\ell_i}\right).
\end{equation}
Moreover, uniformly over histories with $\sum_i k_i\ell_i\le N-1$,
\begin{equation}\label{eq:a-sum-new}
 \sum_i a_0^{\ell_i}=O\left(\frac{N^\eta}{\log N}\right),
 \qquad
 \eta=\log_d(da_0).
\end{equation}
In particular, $\mathcal Q_C$ has no zero in $|z|<\rhoP^{-1}$.
\end{lemma}

\begin{proof}
Since $q_i\le\rhoP^{\ell_i}$, every zero of the $i$th factor in \eqref{eq:code-factor} has modulus $q_i^{-1/\ell_i}\ge\rhoP^{-1}$.  On $|z|=\mathfrak r$, let $r_i=q_i\mathfrak r^{\ell_i}\le a_0^{\ell_i}$.  Then
\begin{equation*}
 |1+q_iz^{\ell_i}+\cdots+(q_iz^{\ell_i})^{k_i}|
 \ge\frac{1-r_i}{1+r_i},
\end{equation*}
which gives \eqref{eq:Q-inverse-new}.

For \eqref{eq:a-sum-new}, let $a_n$ be the number of selected prefixes of length $n$ and take $J=\lfloor\log_dN\rfloor$.  Since $da_0>1$, Lemma~\ref{lem:prefix-count} gives
\begin{align*}
 \sum_{n\le J}a_na_0^n
 &\le \sum_{n\le J}\frac{(da_0)^n}{n}
   +K\sum_{n\le J}(\sqrt d\,a_0)^n\\
 &=O\left(\frac{(da_0)^J}{J}\right),
\end{align*}
because $\sqrt d\,a_0<da_0$.
For $n>J$, the length budget gives
\begin{equation*}
 \sum_{n>J}a_na_0^n
 \le\frac{a_0^{J+1}}{J+1}\sum_{n>J}na_n
 \le\frac{Na_0^{J+1}}{J+1}.
\end{equation*}
Both terms are $O(N^\eta/\log N)$.
\end{proof}

\begin{proof}[Proof of Theorem~\ref{thm:uniform-renewal-new}]
By Lemma~\ref{lem:context-gf-new} and \eqref{eq:code-factor},
\begin{equation}\label{eq:renewal-factor-new}
 \sum_{L\ge0}B(C,f,L)z^L
 =\frac{1-\phi_f(z)}{(1-z)\mathcal Q_C(z)}.
\end{equation}
Remove the pole at $z=1$:
\begin{equation}\label{eq:E-remainder-new}
 E_{C,f}(z):=
 \frac{1-\phi_f(z)}{(1-z)\mathcal Q_C(z)}
 -\frac{1-q_f}{\lbar_C(1-z)}.
\end{equation}
Since $\phi_f(1)=q_f$ and $\mathcal Q_C(1)=\lbar_C$, the numerator after combining the fractions vanishes at $z=1$; hence the singularity is removable.  Lemma~\ref{lem:zero-free-new} shows that $E_{C,f}$ is analytic in $|z|<\rhoP^{-1}$.

For $|z|=\mathfrak r$, one has $|1-z|\ge\mathfrak r-1$, $|1-\phi_f(z)|\le2$, and $1\le\lbar_C\le N$: Proposition~\ref{prop:recovery} gives $\max_{w\in C}|w|\le1+\sum_i k_i|u_i|\le N$, and $\lbar_C$ is a probability-weighted mean length.  Lemma~\ref{lem:zero-free-new} therefore yields
\begin{equation}\label{eq:E-sup-new}
 \sup_{|z|=\mathfrak r}|E_{C,f}(z)|
 \le K\exp\left(K\frac{N^\eta}{\log N}\right)
\end{equation}
uniformly over the final code, the forbidden word, and every generating history within the length budget.

By \eqref{eq:renewal-factor-new}--\eqref{eq:E-remainder-new},
\begin{equation*}
 B(C,f,L)-\frac{1-q_f}{\lbar_C}=[z^L]E_{C,f}(z).
\end{equation*}
Cauchy's coefficient estimate thus gives, for integer $L\ge\omega_N$,
\begin{equation}\label{eq:renewal-coefficient-new}
 \left|B(C,f,L)-\frac{1-q_f}{\lbar_C}\right|
 \le K\exp\left(-L\log\mathfrak r
                 +K\frac{N^\eta}{\log N}\right).
\end{equation}
For nonempty $f$, $q_f\le\rhoP$, while $q_\varnothing=0$; hence the principal term is at least $(1-\rhoP)/N$.  Because $\eta<\vartheta$ and $L\ge N^\vartheta$, division by the principal term in \eqref{eq:renewal-coefficient-new} gives
\begin{equation*}
 |\mathcal H(C,f,L)-1|\le KN e^{-cN^\vartheta}
\end{equation*}
for all large $N$, uniformly in the stated range.  This proves \eqref{eq:beta-def} and \eqref{eq:beta-rate}.
\end{proof}

\section*{Author Contributions and Use of Artificial Intelligence}
\noindent\textbf{Author contributions.}
Conceptualization and formulation of the problem; supervision and guidance of the proof development; validation and critical review of the mathematical arguments; writing--review and editing.  The author reviewed and approved the final manuscript and assumes full responsibility for its correctness and integrity.

\smallskip
\noindent\textbf{Use of artificial intelligence.}
GPT-5.6 Sol Pro was used as an AI-assisted tool for formal analysis and for developing and drafting the mathematical proofs.  All AI-generated mathematical content was reviewed and validated by the author.  GPT-5.6 Sol Pro is not an author and assumes no responsibility for the content of the manuscript.


\begin{thebibliography}{99}

\bibitem{Titchener1984}
M. R. Titchener,
``Digital encoding by means of new T-codes to provide improved data synchronisation and message integrity,''
\emph{IEE Proc. Comput. Digit. Tech.}, vol.~131, no.~4, pp.~151--153, 1984.

\bibitem{Titchener1996}
M. R. Titchener,
``Generalized T-codes: extended construction algorithm for self-synchronizing variable-length codes,''
\emph{IEE Proc. Commun.}, vol.~143, no.~3, pp.~122--128, 1996.

\bibitem{Gunther1998}
U. G\"unther,
\emph{Robust Source Coding with Generalised T-Codes},
Ph.D. dissertation, The University of Auckland, 1998.

\bibitem{NicolescuTitchener1998}
R. Nicolescu and M. R. Titchener,
``Uniqueness theorems for T-codes,''
\emph{Romanian J. Inf. Sci. Technol.}, vol.~1, no.~3, pp.~243--258, 1998.

\bibitem{TitchenerEtAl2005}
M. R. Titchener, R. Nicolescu, L. Staiger, T. A. Gulliver, and U. Speidel,
``Deterministic complexity and entropy,''
\emph{Fundam. Inform.}, vol.~64, no.~1--4, pp.~443--461, 2005.

\bibitem{Kolmogorov1965}
A. N. Kolmogorov,
``Three approaches to the quantitative definition of information,''
\emph{Probl. Inf. Transm.}, vol.~1, no.~1, pp.~1--7, 1965.

\bibitem{LempelZiv1976}
A. Lempel and J. Ziv,
``On the complexity of finite sequences,''
\emph{IEEE Trans. Inf. Theory}, vol.~22, no.~1, pp.~75--81, Jan. 1976.

\bibitem{ZivLempel1978}
J. Ziv and A. Lempel,
``Compression of individual sequences via variable-rate coding,''
\emph{IEEE Trans. Inf. Theory}, vol.~24, no.~5, pp.~530--536, Sep. 1978.

\bibitem{YangSpeidel2005Jucs}
J. Yang and U. Speidel,
``A fast T-decomposition algorithm,''
\emph{J. Universal Comput. Sci.}, vol.~11, no.~6, pp.~1083--1101, 2005,
doi: 10.3217/jucs-011-06-1083.

\bibitem{YangSpeidel2005}
J. Yang and U. Speidel,
``A T-decomposition algorithm with $O(n\log n)$ time and space complexity,''
in \emph{Proc. IEEE Int. Symp. Inf. Theory (ISIT)}, 2005, pp.~23--27,
doi: 10.1109/ISIT.2005.1523285.

\bibitem{RebenichEtAl2014}
N. Rebenich, U. Speidel, S. W. Neville, and T. A. Gulliver,
``FLOTT---A fast, low memory T-transform algorithm for measuring string complexity,''
\emph{IEEE Trans. Comput.}, vol.~63, no.~4, pp.~917--926, Apr. 2014,
doi: 10.1109/TC.2013.29.

\bibitem{Speidel2008}
U. Speidel,
``On the bounds of the Titchener T-complexity,''
in \emph{Proc. 6th Int. Symp. Commun. Syst., Netw. Digit. Signal Process.}, 2008, pp.~321--325.

\bibitem{SpeidelGulliver2012}
U. Speidel and T. A. Gulliver,
``An analytic upper bound on T-complexity,''
in \emph{Proc. IEEE Int. Symp. Inf. Theory (ISIT)}, 2012, pp.~2706--2710.

\bibitem{ClarkTeutsch2015}
G. Clark and J. Teutsch,
``Maximizing T-complexity,''
\emph{Fundam. Inform.}, vol.~139, no.~1, pp.~1--19, 2015.

\bibitem{HamanoYamamoto2009}
K. Hamano and H. Yamamoto,
``A differential equation method to derive the formulas of the T-complexity and the LZ-complexity,''
in \emph{Proc. IEEE Int. Symp. Inf. Theory (ISIT)}, 2009, pp.~625--629.

\bibitem{Hamano2009}
K. Hamano,
\emph{Analysis and Applications of the T-complexity},
Ph.D. dissertation, The University of Tokyo, 2009.

\bibitem{GulliverSpeidel2012Random}
T. A. Gulliver and U. Speidel,
``On the ratio between the maximal T-complexity and the T-complexity of random strings,''
in \emph{Proc. Int. Symp. Inf. Theory Appl. (ISITA)}, 2012, pp.~498--500.

\bibitem{AbramowitzStegun1972}
M. Abramowitz and I. A. Stegun, Eds.,
\emph{Handbook of Mathematical Functions}.
New York, NY, USA: Dover, 1972.

\bibitem{BoucheronLugosiMassart2013}
S. Boucheron, G. Lugosi, and P. Massart,
\emph{Concentration Inequalities: A Nonasymptotic Theory of Independence}.
Oxford, U.K.: Oxford Univ. Press, 2013.

\end{thebibliography}
\end{document}